\documentclass{amsart}
\usepackage{amsmath}
\usepackage{amssymb}
\usepackage{amsthm}
\begin{document}
\def\eq#1{{\rm(\ref{#1})}}
\theoremstyle{plain}
\newtheorem{thm}{Theorem}[section]
\newtheorem{lem}[thm]{Lemma}
\newtheorem{prop}[thm]{Proposition}
\newtheorem{cor}[thm]{Corollary}
\theoremstyle{definition}
\newtheorem{dfn}[thm]{Definition}
\newtheorem{rem}[thm]{Remark}
\def\Ker{\mathop{\rm Ker}}
\def\Coker{\mathop{\rm Coker}}
\def\ind{\mathop{\rm ind}}
\def\Re{\mathop{\rm Re}}
\def\vol{\mathop{\rm vol}}
\def\SO{\mathbin{\rm SO}}
\def\Im{\mathop{\rm Im}}
\def\min{\mathop{\rm min}}
\def\Spec{\mathop{\rm Spec}\nolimits}
\def\Hol{{\textstyle\mathop{\rm Hol}}}
\def\ge{\geqslant}
\def\le{\leqslant}
\def\Z{{\mathbin{\mathbb Z}}}
\def\R{{\mathbin{\mathbb R}}}
\def\N{{\mathbin{\mathbb N}}}
\def\E{{\mathbb E}}
\def\V{{\mathbb V}}
\def\al{\alpha}
\def\be{\beta}
\def\ga{\gamma}
\def\de{\delta}
\def\ep{\epsilon}
\def\io{\iota}
\def\ka{\kappa}
\def\la{\lambda}
\def\ze{\zeta}
\def\th{\theta}
\def\vp{\varphi}
\def\si{\sigma}
\def\up{\upsilon}
\def\om{\omega}
\def\De{\Delta}
\def\Ga{\Gamma}
\def\Th{\Theta}
\def\La{\Lambda}
\def\Om{\Omega}
\def\ts{\textstyle}
\def\sst{\scriptscriptstyle}
\def\sm{\setminus}
\def\op{\oplus}
\def\ot{\otimes}
\def\bigop{\bigoplus}
\def\iy{\infty}
\def\ra{\rightarrow}
\def\longra{\longrightarrow}
\def\dashra{\dashrightarrow}
\def\t{\times}
\def\w{\wedge}
\def\d{{\rm d}}
\def\bs{\boldsymbol}
\def\ci{\circ}
\def\ti{\tilde}
\def\ov{\overline}
\def\md#1{\vert #1 \vert}
\def\nm#1{\Vert #1 \Vert}
\def\bmd#1{\big\vert #1 \big\vert}
\def\cnm#1#2{\Vert #1 \Vert_{C^{#2}}} 
\def\lnm#1#2{\Vert #1 \Vert_{L^{#2}}} 
\def\bnm#1{\bigl\Vert #1 \bigr\Vert}
\def\bcnm#1#2{\bigl\Vert #1 \bigr\Vert_{C^{#2}}} 
\def\blnm#1#2{\bigl\Vert #1 \bigr\Vert_{L^{#2}}} 
\title[Deformations of HL and RS manifolds]{Deformations of Lagrangian Type Submanifolds inside $G_2$ manifolds}
\author{Rebecca Glover and Sema Salur}
\dedicatory{Dedicated to the memory of Ruth I. Michler}
\thanks{S.Salur is partially supported by NSF grant 1105663}
\address{Department  of Mathematics, University of Rochester, Rochester, NY, 14627}
\email{rglover3@ur.rochester.edu }
\address {Department of Mathematics, University of Rochester, Rochester, NY, 14627 }
\email{sema.salur@rochester.edu } 
\date{\today}

\begin{abstract} 3-dimensional Harvey Lawson submanifolds were introduced in an earlier paper by Akbulut-Salur, \cite{AS1} as examples of Lagrangian-type manifolds inside $G_2$ manifold. In this paper, we first show that the space of deformations of a smooth, compact, orientable Harvey-Lawson submanifold $HL$ in a $G_2$ manifold $M$ can be identified with the direct sum of the space of smooth functions and closed 2-forms on $HL$. We then introduce a new class of Lagrangian-type 4-dimensional submanifolds inside $G_2$, call them RS submanifolds and prove that the space of deformations of a smooth, compact, orientable $RS$ submanifold in a $G_2$ manifold $M$ can be identified with closed 3-forms on $RS$. 
\end{abstract}
\date{}
\maketitle
\section{Introduction}

It is well-known that a smooth symplectic manifold $N^{2n}$ is equipped with a closed, nondegenerate differential 2-form $\omega$.  An $n$-dimensional submanifold $L^n$ of $N^{2n}$ is called Lagrangian if the restriction of $\omega$ to $L$ is zero. Lagrangian submanifolds and their deformations have important applications in symplectic geometry and mathematical physics. In particular, they play a role in establishing the correspondence between ``Calabi-Yau mirror pairs'' in string theory via the Fukaya category.  

\vspace{.1in} 

Let $(M,\varphi)$ be a $G_2$ manifold with calibration 3-form $\varphi$.  As with the symplectic 2-form $\omega$, the calibration 3-form $\varphi$ on a $G_2$ manifold is closed and nondegenerate.  Moreover, the Hodge dual, $\star\varphi$, of $\varphi$ is also closed and nondegenerate. Therefore $G_2$ manifolds provide a natural setting in which one can search for Lagrangian-type submanifolds that correspond to $\varphi$ and $\star\varphi$.  We call them Harvey-Lawson (HL)  and RS manifolds, respectively.  For more on the definition and properties of HL manifolds we refer the reader to \cite {AS1}.  One can study these submanifolds and their deformation spaces and moreover, search for relations by using Fukaya categories.  Through these results, we hope to use the relationship between Calabi-Yau $3$-folds and $G_2$ manifolds to answer questions about the mirror symmetry of Calabi-Yau $3$-folds and existence of calibrated submanifolds, \cite{AS2},  \cite{AS3}, \cite{AS4}, \cite{bryant2}, \cite{RS2}, \cite{salur1}.  This will be our motivation for the future directions, \cite{YS}, \cite{RS1}.

\vspace{.1in}

In this paper we define these Lagrangian-type submanifolds of a $G_2$ manifold $M$.  We then describe some properties of these objects and prove the following theorems.

\begin{thm} The space of infinitesimal deformations of a smooth, compact,
orientable 3-dimensional Harvey-Lawson submanifold $HL$ in a $G_2$
manifold $M$ within the class of $HL$ submanifolds is infinite-dimensional. The deformation space can be identified with the direct sum of the spaces of smooth functions and closed differential 2-forms on $HL$.
\end{thm}

\begin{thm} The space of all infinitesimal deformations of a smooth, compact,
orientable 4-dimensional $RS$ submanifold in a $G_2$
manifold $M$ within the class of $RS$ submanifolds is infinite-dimensional and can be identified with closed differential 3-forms on $RS$.
\end{thm}

\begin{rem}
Note that $M$ does not have to be a manifold with $G_2$ holonomy for these theorems.  Theorem 1.1 also holds when $M$ is a manifold with  a closed $G_2$ structure $\phi$ and Theorem 1.2 can be extended to the case when $M$ is a manifold with a co-closed $G_2$ structure $\phi$. 
\end{rem}

\section{Deformations of $HL$ submanifolds}\label{HLsub}

In this section, we study infinitesimal deformations of $HL$ submanifolds. First, let's recall some basic definitions.  For more details about $G_2$ manifolds, we refer the reader to \cite{bryant3}, \cite{bryant4}, and \cite{HL}.

\vspace{.1in}

A manifold with $G_2$ structure is a smooth seven-dimensional manifold $M$ such that the structure group of $M$ reduces to the exceptional Lie group $G_2$.  Equivalently, since $G_2$ can be defined as the automorphism group of $\mathbb{R}^7$ that preserves the $3$-form 
\[ \varphi_0 = e^{123} + e^{145} + e^{167} + e^{246} - e^{257} - e^{347} - e^{356} \]
where $e^{123} = dx^1 \wedge dx^2 \wedge dx^3$ and $(x_1, \ldots, x_7)$ are coordinates on $\mathbb{R}^7$,
we can define a $G_2$ structure in the following manner.

\begin{dfn} A {\it manifold with $G_2$ structure} is a smooth $7$-dimensional manifold $M$ equipped with a nondegenerate $3$-form $\varphi\in \Omega^3(M)$ such that at any point $p\in M$, 
\[ (T_p M, \varphi_p) \cong (\mathbb{R}^7, \varphi_0). \]
\end{dfn}

\vspace{.1in}

This definition implies that on a local chart of a manifold with $G_2$ structure $(M, \varphi)$, up to quadratic terms the 3-form $\varphi $ coincides with the form $\varphi_{0} \in \Omega^{3} (\R^7)$.  The 3-form $\varphi $ determines a metric $g$ and a cross product $\times$ on $M$ given by
\[ \varphi (u,v,w) = g(u, v\times w), \  \  u, v, w  \in TM. \]
In this paper, we write this metric as $g(\cdot, \cdot) = \langle \cdot, \cdot \rangle$ for simplicity.

\begin{dfn} Suppose $(M, \varphi)$ is a manifold with $G_2$ structure.  We call $(M, \varphi)$ a {\it $G_2$ manifold} if $\varphi$ is torsion-free with respect to the Levi-Civita connection.
Equivalently, we could say that $M$ is a $G_2$ manifold if $M$ has holonomy contained in the Lie group $G_2$. 
\end{dfn}

The torsion-free condition is equivalent to the condition that the form $\varphi$ is closed and co-closed, i.e.
\[ d\varphi = d\star\varphi = 0 .\]  

We call $\varphi$ and $\star\varphi$ the calibration $3$-form and $4$-form, respectively, as they define {\it calibrated submanifolds}; manifolds that are volume-minimizing in their homology class.

\begin{dfn} Let $(M, \varphi )$ be a $G_2$ manifold with calibration 3-form $\varphi$. A 4-dimensional
submanifold $C\subset M$ is called {\em coassociative } if
$\varphi|_C=0$. A 3-dimensional submanifold $A\subset M$ is called
{\em associative} if $\varphi|_A\equiv dvol(A)$.

\vspace{.1in}

Note that the condition $\varphi|_A\equiv dvol(A)$is equivalent to the condition that $\chi|_A\equiv 0$,  where $\chi \in \Omega^{3}(M,TM)$ is the tangent bundle-valued 3-form defined by the identity
\begin{equation*}
\langle \chi (u,v,w) , z \rangle=*\varphi  (u,v,w,z).
\end{equation*}
\end{dfn}

\begin{rem}
The equivalence of $\varphi|_{A} = dvol$ and $\langle \chi|_{A}, \chi|_{A} \rangle = 0$ follows from the associator equality which was shown in Harvey-Lawson, \cite{HL}:

\begin{equation*}
\varphi  (u,v,w)^2 + \frac{1}{4}|\chi (u,v,w)|^2= |u\wedge v\wedge w|^2 .
\end{equation*}

\end{rem}

\begin{dfn}
 A {\it Harvey-Lawson manifold} is a 3-dimensional submanifold $HL\subset M$ of a $G_2$ manifold such that 
\[ \varphi|_{HL} = 0. \]

Equivalently, this is defined by $\langle \chi|_{HL}, \chi|_{HL} \rangle = 1$. 
\end{dfn}

\begin{rem}

Again, the associator equality, \cite{HL},  gives the equivalence of the conditions $\varphi|_{HL} = 0$ and $\langle \chi|_{HL}, \chi|_{HL} \rangle = 1$.

\end{rem}

\begin{rem}
One can obtain $G_2$ manifolds from Calabi-Yau manifolds in the following way. Let $(N,\omega, \Omega)$ be a complex
3-dimensional Calabi-Yau manifold with K\"{a}hler form $\omega$
and a nonvanishing holomorphic (3,0)-form $\Omega$.  Then the direct product
$N^6\times S^1$ has holonomy group $SU(3)$ which is a subset of $G_2$. Therefore $N^6\times S^1$  is a
$G_2$ manifold. In this particular case, $ \varphi = \Re \Omega + \omega \wedge dt$. For the noncompact case $N^6\times \mathbb{R}$ is also a $G_2$ manifold (with reduced holonomy). Using these structures, it is easy to show that all special Lagrangian submanifolds (with phase $\theta=\frac{\pi}{2}$) of the Calabi-Yau manifold $N$ will be Harvey-Lawson submanifolds of $N^6\times S^1$ or 
$N^6\times \mathbb{R}$.

\end{rem}

Similar to the tangent bundle-valued 3-form $\chi$ on a $G_2$ manifold, there is also a tangent
bundle-valued 2-form, which is just the cross product of $M$.  

\vspace{.05in}

\begin{dfn} Let $(M, \varphi )$ be a $G_2$ manifold. Then $\psi \in
\Omega^{2}(M, TM)$ is the tangent bundle-valued 2-form defined by
the identity
\begin{equation*}
\langle \psi (u,v) , w \rangle=\varphi  (u,v,w)=\langle u\times v , w
\rangle .
\end{equation*}
\end{dfn}

\vspace{.05in}

We can also express the tangent bundle-valued forms $\chi$ and $\psi$ in local coordinates as above for $\varphi$. More generally, if $e_1,...,e_7$ is any local orthonormal frame with dual frame $e^1,..., e^7$, by definition we can write $\chi$ and $\psi$ in coordinates as

 \begin{equation*}
\begin{aligned}
\chi=&\;\;(e^{256}+e^{247}+e^{346}-e^{357})e_1\\
&+(-e^{156}-e^{147}-e^{345}-e^{367})e_2\\
&+(e^{157}-e^{146}+e^{245}+e^{267})e_3\\
&+(e^{127}+e^{136}-e^{235}-e^{567})e_4\\
&+(e^{126}-e^{137}+e^{234}+e^{467})e_5\\
&+(-e^{125}-e^{134}-e^{237}-e^{457})e_6\\
&+(-e^{124}+e^{135}+e^{236}+e^{456})e_7,\\
\end{aligned}
\end{equation*}

\vspace{.1in}

\begin{equation*}
\begin{aligned}
\psi=&\;\;(e^{23}+e^{45}+e^{67})e_1\\
&+(e^{46}-e^{57}-e^{13})e_2\\
&+(e^{12}-e^{47}-e^{56})e_3\\
&+(e^{37}-e^{15}-e^{26})e_4\\
&+(e^{14}+e^{27}+e^{36})e_5\\
&+(e^{24}-e^{17}-e^{35})e_6\\
&+(e^{16}-e^{25}-e^{34})e_7.\\
\end{aligned}
\end{equation*}

Next, we study the deformations of HL submanifolds in a $G_2$ manifold. Recall that McLean studied the deformations of compact special Lagrangian
submanifolds in Calabi-Yau manifolds and proved the following theorem \cite{mclean}.

\begin{thm}  The moduli space of all deformations of a smooth, compact,
orientable special Lagrangian submanifold $L$ in a Calabi-Yau
manifold $N$ within the class of special Lagrangian submanifolds
is a smooth manifold of dimension equal to dim$(H^1(L))$.
\label{co2thm}
\end{thm}

\vspace{.1in}
In this section, we provide the analogue of this theorem for HL manifolds (and also for $RS$ manifolds in the next section), but in these cases, the space of deformations will be infinite-dimensional.  Note that this is similar to the space of deformations of Lagrangian submanifolds of symplectic manifolds.

\vspace{.1in}

We first describe the normal bundle of an $HL$ submanifold inside $G_2$ manifold $M$. Recall that an orthonormal $3$-frame field $\Gamma=<u,v,w>$ on $(M,\varphi )$ is called a {\it $G_2$-frame field} if  $\varphi(u,v,w)= <u\times v, w>=0$, \cite{AS1}. 
By results of Emery Thomas, \cite{ET}, there exists a nonvanishing $2$-frame field $\Lambda=<u,v>$ on $M$. 
 
 \vspace{.1in}
 
 Let $TM= {\E}\oplus {\V}$  be the corresponding splitting with ${\E}=<u,v, u\times v>$.  Let $w$ be a unit section of the bundle ${\V} \to M$. Even though this section $w$ may not exist on the entire manifold $M$, it exists on a tubular neighbourhood of the $3$-skeleton $M^{(3)}$ of $M$ (by obstruction theory) which is the complement of a $3$ complex $Y\subset M$. Therefore, $\varphi(u,v, w)=<u\times v, w>=0$, and hence $\Gamma=<u,v,w>$ is a $G_{2}$ (i.e. $HL$) frame field.

\vspace{.1in}

Consider the non-vanishing vector field
$$R=\chi(u,v,w)=-u\times(v\times w) .$$
We recall the following lemma from \cite{AS1}.

\begin{lem} The following properties hold:

\begin{itemize}
\item[(a)]  If $< u,v,w>$ is an HL $3$-plane field, then $\V=< u,v,w, R>$ is a coassociative $4$-plane field. 
\item[(b)] $ \E= < u\times v, v\times w, w\times u > $ is an associative $3$-plane field.
\item[(c)] $\E \perp \V$.
 \item[(d)] $\{u,v,w, R, u\times v, v\times w, w\times u\}$ is an orthonormal frame field on $M$.
\end{itemize}
\end{lem}

 \vspace{.1in}

In particular we can express $\varphi$ as

$$\varphi = u^{\#}\wedge v^{\#}\wedge (u^{\#}\times v^{\#}) + v^{\#}\wedge w^{\#} \wedge (v^{\#} \times w^{\#}) + w^{\#} \wedge u^{\#} \wedge (w^{\#} \times u^{\#}) $$
$$\;\;\;\;\;\;\; + \;u^{\#} \wedge R^{\#} \wedge (v^{\#}\times w^{\#}) + v^{\#}\wedge R^{\#} \wedge (w^{\#} \times u^{\#}) + w^{\#} \wedge R^{\#} \wedge (u^{\#} \times v^{\#}) $$
$$ -\; (u^{\#}\times v^{\#}) \wedge (v^{\#}\times w^{\#} ) \wedge (w^{\#}\times u^{\#}) .$$

\vspace{.1in}

Note that $\{ u,v,w, R, u\times v, v\times w, w\times u\}$  and $\{u,v,w, R, w\times R, u\times R, v\times R\}$ are equivalent frames, therefore $\varphi$  can be given by

$$\varphi = u^{\#}\wedge v^{\#}\wedge (w^{\#}\times R^{\#}) + v^{\#}\wedge w^{\#} \wedge (u^{\#} \times R^{\#}) + w^{\#} \wedge u^{\#} \wedge (v^{\#} \times R^{\#}) $$
$$\;\;\;\;\;\;\; + \;u^{\#} \wedge R^{\#} \wedge (u^{\#} \times R^{\#}) + v^{\#}\wedge R^{\#} \wedge (v^{\#} \times R^{\#})  + w^{\#} \wedge R^{\#} \wedge (w^{\#}\times R^{\#}) $$
$$ -\; (w^{\#}\times R^{\#})\wedge  (u^{\#} \times R^{\#}) \wedge (v^{\#} \times R^{\#})  .$$

\vspace{.1in}

The normal bundle of an $HL$ submanifold can be decomposed as 
\[ N(HL) = \tilde{N}(HL)\oplus R ,\]
 where $\tilde{N}(HL)$ is generated by vector fields $u\times R$, $v\times R$, and $w\times R$.  Since $\langle \psi (u,v) , w \rangle=\varphi  (u,v,w)=\langle u\times v , w \rangle=0$ for an $HL$ submanifold,  $\tilde{N}(HL)$ is isomoprhic to $T(HL)$.  The cross product structure $\times$ (also known as $\psi$) induces this isomorphism.

\vspace{.1in}

\begin{thm} The space of infinitesimal deformations of a smooth, compact,
orientable 3-dimensional Harvey Lawson submanifold $HL$ in a $G_2$
manifold $M$ within the class of $HL$ submanifolds is infinite-dimensional. The deformation space can be identified with the direct sum of the spaces of smooth functions and closed 2-forms on $HL$.
\label{thm1}
\end{thm}

\begin{proof} For a small vector field $V$ the deformation map
is a map, $F$, defined from the space of sections of the normal bundle,
$\Gamma(N(HL))$, to the space of differential $3$-forms,
$\Lambda^3T^*(HL)$ on $HL$, such that
\begin{equation*}
\begin{split}
&F: \Gamma(N(HL))\rightarrow
\Lambda^3T^*(HL) , \\
&F(V)=((\exp_V)^*(\varphi|_{HL_V})).
\end{split}
\end{equation*}

\vspace{.1in}

The deformation map $F$ is the restriction of $\varphi$ to $HL_{V}$ and then pulled back to $HL$ via $(\exp_V)^*$ where
$\exp_V$ is the normal exponential map which gives a diffeomorphism of
$HL$ onto its image $HL_V$ in a neighborhood of 0.

\vspace{.1in}
There is a natural identification of normal vector fields to $HL$ with differential
$1$-forms on $HL$. Furthermore, since $HL$ is compact, these normal
vector fields can be identified with nearby submanifolds. Under
these identifications, the kernel of $F$ then corresponds to the $HL$ deformations.

\vspace{.1in}

The linearization of $F$ at $(0)$ is given by
\begin{equation*}
dF(0):\Gamma(N(HL))\rightarrow \Lambda^3T^*(HL)
\end{equation*}
\noindent where
\begin{equation*}
\begin{split}
dF(0)(V)=&\frac{\displaystyle\partial}{\displaystyle\partial{t}}F(tV)|_{t=0}=\frac{\displaystyle\partial}{\displaystyle\partial{t}}[\exp_{tV}^*(\varphi)] \\
=&[{\mathcal L}_{V}(\varphi)|_{HL}].
\end{split}
\end{equation*}

\vspace{.1in}

Further, by Cartan's formula, we have
\begin{equation*}
\begin{split}
dF(0)(V)&=((i_Vd\varphi +d(i_V\varphi))|_{HL}\\
&=d(i_V\varphi)|_{HL}=(d\star v),
\end{split}
\end{equation*}
where $i_V$ represents the interior derivative, $v$ is the dual
$1$-form to the vector field $V$ with respect to the induced
metric, and $\star v$ is the Hodge dual of $v$ on HL. Hence
\begin{equation*}
dF(0)(V) =(d\star v)=(d^*v).
\end{equation*}


Therefore the space of nontrivial deformations of $HL$ submanifolds can be identified with closed $2$-forms on $HL$. The other component of the deformation space comes from the trivial deformations of $HL$ submanifolds. These correspond to deforming a 3-dimensional $HL$ manifold inside a coassociative manifold. By definition, any such 3-manifold will be $HL$, which implies that deformations of $HL$ inside a coassociative submanifold in the direction of $R$ can be identified with smooth functions on $HL$.

\end{proof}

\section{Deformations of RS submanifolds}
Let $(M,\varphi)$ be a $G_2$ manifold.  In this section, we define Lagrangian-type $4$-dimensional submanifolds of $M$ and describe their deformation space.
Note that we can write the calibration $4$-form $\star\varphi$, in local coordinates as
\[ \star\varphi = e^{4567} + e^{2367} + e^{2345} + e^{1357} - e^{1346} - e^{1256} - e^{1247} .\]

\begin{dfn}
An {\it RS manifold} is a 4-dimensional submanifold $RS\subset M$ of a $G_2$ manifold such that
\[ \star\varphi|_{RS} = 0 . \]
As with Harvey-Lawson manifolds, there is a tangent bundle-valued $4$-form $\sigma$ defining $RS$ such that
\[ \star\varphi|_{RS} = 0 \]
if and only if 
\[ \langle \sigma|_{RS}, \sigma|_{RS} \rangle = 1. \]

\end{dfn}

This form, $\sigma$, is called the {\it coassociator}, and we can write it as
\[ \sigma(u,v,w,z) = \langle v, w\times z \rangle u + \langle w, u\times z \rangle v + \langle u, v\times z \rangle w + \langle v, u\times w \rangle z .\]
In \cite{HL}, Harvey and Lawson prove the coassociator equality,
\[ *\varphi(u,v,w,z)^2 + \frac{1}{4} | \sigma(u,v,w,z)|^2 = |u\wedge v\wedge w \wedge z|^2 .\]
Note that this identity provides justification for the above definition.
Using the coordinates above for $\varphi$, we can write $\sigma$ in coordinates as follows:

\begin{eqnarray*} \sigma &= &(-e^{1347}-e^{1356} - e^{1257} + e^{1246})e_1 \\
&  & + (-e^{2347} - e^{2356} - e^{1267} - e^{1245})e_2 \\
&  & + (-e^{2346} + e^{2357} - e^{1367} - e^{1345})e_3 \\
&  & +  (e^{3456} +e^{2457} - e^{1467} - e^{1234})e_4 \\
&  & + (-e^{3457} + e^{2456} - e^{1567} - e^{1235})e_5  \\
&  & + (-e^{3467}- e^{2567}- e^{1456}- e^{1236})e_6 \\
&  & + (e^{3567} - e^{2467} - e^{1457} - e^{1237}) e_7. \end{eqnarray*}

\vspace{.1in}

\begin{rem}
Recall that the direct product $N^6\times S^1$ is a $G_2$ manifold, where $(N,\omega, \Omega)$ is a complex
3-dimensional Calabi-Yau manifold with K\"{a}hler form $\omega$ and nonvanishing holomorphic $(3,0)$-form $\Omega$.  In this case, $ \varphi = \Re \Omega + \omega \wedge dt$ and $* \varphi = -dt\wedge \Im \Omega + \frac{1}{2}(\omega \wedge \omega)$. Therefore, the 4-manifolds given by $SL\times S^1$, where $SL$ is a special Lagrangian submanifold of $N$ with phase $\theta=0$ will be $RS$ submanifolds of $N^6\times S^1$.  Analogous results hold for the noncompact $G_2$ manifold $N^6\times \mathbb{R}$.

\end{rem}

Given a nonvanishing $4$-frame field $\langle u,v,w,z \rangle$, note that we can define a nonvanishing vector field $S$ given by
$$S=\sigma (u,v,w,z)=\langle v, w\times z \rangle u + \langle w, u\times z \rangle v + \langle u, v\times z \rangle w + \langle v, u\times w \rangle z .$$

Using

\begin{equation*}
\langle \chi (u,v,w) , z \rangle=*\varphi  (u,v,w,z)
\end{equation*}
and
\begin{equation*}
\langle \psi (u,v) , w \rangle=\varphi  (u,v,w)=\langle u\times v , w
\rangle
\end{equation*}
we have the following lemma.  Note that given a nonvanishing $2$-frame field $<u,v>$, we can extend it by some nonvanishing vector field $w$.  Then $\tilde{S} := \sigma (u,v,u\times v, w)$ is a nonvanishing vector field on $M$.

\begin{lem} Let $\tilde{S}=\sigma (u,v,u\times v, w)$ for nonvanishing vector fields $u,v,w\in TM$. Then the following properties hold:
\begin{itemize}
\item[(a)]  $<u,v,u\times v,w>$ is an $RS$ $4$-plane field.
\item[(b)] $<u,v,u\times v, \tilde{S}>$ is an $RS$ $4$-plane field. 
\item[(c)] $<\tilde{S}, u\times \tilde{S}, v\times  \tilde{S}, (u\times v)\times  \tilde{S}>$ is a coassociative $4$-plane field.
 \item[(d)] $<u\times  \tilde{S},  v\times \tilde{S}, (u\times v)\times  \tilde{S} >$ is an HL $3$-plane field.
 \item[(e)] $\{u,v,u\times v,  \tilde{S}, u\times  \tilde{S},  v\times  \tilde{S}, (u\times v)\times  \tilde{S}\}$ is an orthonormal frame field.
 
\end{itemize}
\end{lem}

The proof of this lemma is similar to that for HL manifolds, which can be found in \cite{AS1}.  
Using these structures we can express the 3-form $\varphi$ as

$$\varphi = u^{\#}\wedge v^{\#}\wedge ((u\times v)^{\#}\times S^{\#}) + v^{\#}\wedge (u\times v)^{\#} \wedge (u\times S)^{\#} + (u\times v)^{\#} \wedge u^{\#} \wedge (v^{\#} \times S^{\#}) $$
$$\;\;\;\;\;\;\; + \;u^{\#} \wedge S^{\#} \wedge (u^{\#} \times S^{\#}) + v^{\#}\wedge S^{\#} \wedge (v^{\#} \times S^{\#})  + (u\times v)^{\#} \wedge S^{\#} \wedge ((u\times v)^{\#}\times S^{\#}) $$
$$ -\; ((u\times v)^{\#}\times S^{\#})\wedge  (u^{\#} \times S^{\#}) \wedge (v^{\#} \times S^{\#})  .$$

\vspace{.1in}

The tangent bundle of an $RS$ manifold decomposes as 
\[ T(RS)=\tilde{T}(RS)\oplus \tilde{S} ,\] where $\tilde{T}(RS)$ is generated by vectors $\tilde{S}\times u$, $\tilde{S}\times v$, $\tilde{S}\times w$. Since 
\[\langle \chi (u,v,w) , z \rangle=*\varphi  (u,v,w,z)=0\]
  for an $RS$ submanifold, its normal bundle $N(RS)$ is isomorphic to $\tilde{T}(RS)$. The bundle-valued 3-form $\chi$ induces this isomorphism. 

\vspace{.1in}

Under these identifications, we can prove the following theorem.

\begin{thm} The space of all infinitesimal deformations of a smooth, compact,
orientable 4-dimensional $RS$ submanifold in a $G_2$
manifold $M$ within the class of $RS$ submanifolds is infinite-dimensional and can be identified with closed differential 3-forms on $RS$.
\label{thm2}
\end{thm}

\begin{proof} As in the proof of Theorem \ref{thm1}, for a small vector field $V$, the deformation map $F$ is defined as
\begin{equation*}
\begin{split}
&F: \Gamma(N(RS))\rightarrow
\Lambda^4T^*(RS)\\
&F(V)=(\exp_V)^*(\star\varphi)
\end{split}
\end{equation*}

In other words, the deformation map $F$ is the restriction of $\star\varphi$
to $RS_V$ and then pulled back to $RS$ via $(\exp_V)^*$ where
$\exp_V$ is the exponential map which gives a diffeomorphism of
$RS$ onto its image $RS_V$ in a neighborhood of 0.  For reasons analogous to those for HL manifolds, the kernel of $F$ corresponds to deformations of $RS$ manifolds.



\vspace{.1in}

Furthermore, we can write the linearization of $F$ at $(0)$ as
\begin{equation*}
dF(0):\Gamma(N(RS))\rightarrow \Lambda^4T^*(RS)
\end{equation*}
\noindent where
\begin{equation*}
\begin{split}
dF(0)(V)=&[-{\mathcal L}_{V}(*\varphi)|_{RS}] \\ 
=& d(i_V(*\varphi)))|_{RS}=d\star v
\end{split}
\end{equation*}

Here again, $i_V$ represents the interior derivative and $v$ is the dual
$1$-form to the vector field $V$ with respect to the induced
metric. Hence
\begin{equation*}
dF(0)(V) =d\star v=\star(d^*v).
\end{equation*}

Therefore the space of infinitesimal deformations of $RS$ manifolds can be identified with closed 3 forms on $RS$.

\end{proof}

\begin{rem}  Recall that for a Lagrangian submanifold $L$ of $(N^{2n}, \omega)$, there exists an almost complex structure $J$ on $N$ such that $J$ maps vectors on $L$ to vectors orthogonal to $L$.  In other words, 
\[ \langle Ju, v \rangle = 0 \]
for all $u, v\in TL$.  As a final remark, we note that we have an analogous situation here for HL and RS manifolds.  In the Harvey-Lawson case, the cross product $\times$ acts as this complex structure.  Further, in the case of RS manifolds, the triple cross product assumes the role of the complex structure, taking tangent vectors on the RS manifold to normal vectors.
\end{rem}

{\small{\it Acknowledgements.} Part of this work was done when the second author was visiting Cornell University as the Ruth I. Michler Fellow during the Spring of 2015. Many thanks to the mathematics department at Cornell for their hospitality and AWM for the support during the course of this work. Special thanks are to Dogan Ertan 
for all his help and encouragement. }

\end{document}